\documentclass[12pt]{amsart}
\allowdisplaybreaks
\usepackage{eucal}
\newtheorem{theorem}{Theorem}
\newtheorem{lemma}[theorem]{Lemma}

\newtheorem{cor}[theorem]{Corollary}
\theoremstyle{definition}
\newtheorem{defn}[theorem]{Definition}

\newtheorem{rks}[theorem]{Remarks}

\numberwithin{theorem}{section}
\newcommand{\End}{\mathsf{End}}
\newcommand{\Hom}{\mathsf{Hom}}
\newcommand{\F}{\mathbb{F}}

\newcommand{\Sym}{\mathfrak{S}}
\newcommand{\Z}{\mathbb{Z}}

\newcommand{\C}{\mathbb{C}}
\newcommand{\ii}{\mathbf{i}}
\newcommand{\jj}{\mathbf{j}}
\newcommand{\im}{\operatorname{im}}

\title{Schur--Weyl duality over finite fields}
\author{David Benson}
\address{Department of Mathematics, University of Aberdeen, 
Aberdeen AB24 3EA, Scotland}
\email{bensondj@maths.abdn.ac.uk}
\author{Stephen Doty}
\address{Department of Mathematics, Loyola University Chicago, 
Chicago IL 60626, USA}
\email{doty@math.luc.edu}
\thanks{Both authors would like to thank MSRI for its hospitality 
while this work was in progress}
\begin{document}

\begin{abstract}
We prove a version of Schur--Weyl duality over finite fields.  We
prove that for any field $k$, if $k$ has at least $r+1$ elements, then
Schur--Weyl duality holds for the $r$th tensor power of a finite
dimensional vector space $V$.  Moreover, if the dimension of $V$ is at
least $r+1$, the natural map $k\Sym_r \to \End_{GL(V)}(V^{\otimes r})$
is an isomorphism. This isomorphism may fail if $\dim_k V$ is not
strictly larger than $r$.
\end{abstract}

\maketitle

\section{Introduction}\noindent
Let $k$ be a field and let $V$ be an $n$-dimensional vector space over
$k$. We identify $G=GL(V)$ with $GL(n,k)$ and $V$ with $k^n$ as usual,
by fixing a basis $\{v_1, \dots, v_n\}$ of $V$.  $G$ acts on $V$ in
the natural way, and thus on the tensor product space $V^{\otimes
  r}$. Moreover, $V^{\otimes r}$ admits an action of the symmetric
group $\Sym_r$ permuting the tensor multiplicands. For $\sigma\in
\Sym_r$, the action is given by ``place permutation''
\[ \sigma(x_1\otimes\dots\otimes x_r) = x_{\sigma^{-1}(1)} \otimes
\dots \otimes x_{\sigma^{-1}(r)}, \] extended linearly to
$k\Sym_r$. One observes that the actions of $G$ and $\Sym_r$ commute,
so there are natural algebra homomorphisms
\begin{equation} \label{eq:maps}
\rho\colon k\Sym_r \to \End_G(V^{\otimes r}); \qquad
\varphi\colon kG \to \End_{\Sym_r}(V^{\otimes r}). 
\end{equation} 
In case $k =\C$, Schur \cite{Schur:1927a} proved that both maps are
surjective, so the image of each map is the full centralizer algebra
for the other action. In particular, the ring of invariants
\[ \End_G(V^{\otimes r}) \cong (V^{\otimes r} \otimes (V^*)^{\otimes r})^G \] 
is generated by the action of the symmetric group.

Assume that $k$ is infinite. In that case, Theorem 4.1 of De Concini
and Procesi \cite{deConcini/Procesi:1976a} shows the surjectivity of
$\rho$; in fact, their result is even more general. In the special
case $n \ge r$, the surjectivity of $\rho$ was treated earlier in
Lemma 3.1 of \cite{Carter/Lusztig:1974a}, where it is in fact an
isomorphism. The surjectivity of $\varphi$ for $k$ infinite follows
immediately from (2.4b) and (2.6c) of J.~A.~Green's monograph
\cite{Green:2007a}. 

The purpose of this paper is to discuss the question of the
surjectivity of the maps $\rho, \varphi$ in \eqref{eq:maps} in case
$k$ is a finite field. Our main results are obtained in Theorem
\ref{th:main1} and Corollary \ref{cor:main2}. Theorem \ref{th:main1}
states that $\rho$ is an isomorphism provided that $\dim_k V > r$, a
result that holds independently of the field $k$. This generalizes a
result of Carter and Lusztig. To obtain Corollary \ref{cor:main2},
which extends Schur--Weyl duality to the case where $k$ has at least
$r+1$ elements, we need only the surjectivity of $\rho$, $\varphi$ in
case $k$ is algebraically closed. Our approach also gives information
about a version of Schur--Weyl duality in which $GL(V)$ is replaced
with $SL(V)$. We note that Thrall \cite{Thrall} also obtained the
surjectivity of $\rho$ in case $k$ is sufficiently large, although the
stated bound on the cardinality of the field is incorrect. The methods
of this paper are different than those of \cite{Thrall}. 

Several remarks are in order. First, in this paper we regard $G =
GL(n,k)$ as a linear group and not as a group scheme.  To distinguish
these objects notationally, we denote by $\mathbf{GL}_{n,k}$ the
algebraic $k$-group scheme $A \mapsto GL(n,A)$ regarded as a
representable functor from commutative $k$-algebras to groups.  If the
ground field $k$ is understood we may write $\mathbf{GL}_n$ instead of
$\mathbf{GL}_{n,k}$. When $k$ is infinite the distinction between
$GL(n,k)$ and $\mathbf{GL}_{n,k}$ makes little difference, since the
algebraic group $GL(n, \overline{k})$ may be identified with the
algebraic $\overline{k}$-group scheme $\mathbf{GL}_{n,\overline{k}}$,
and the group $GL(n, k)$ of $k$-rational points in $GL(n,
\overline{k})$ may be identified with the algebraic $k$-group scheme
$\mathbf{GL}_{n,k}$.  However, in case $k$ is finite the algebraic
$k$-group scheme $\mathbf{GL}_{n,k}$ is a completely different object
from the finite group $G = GL(n,k) = \mathbf{GL}_{n,k}(k)$.

Our second remark is that in the case when $k$ is infinite a sensible
way to study the $kG$-module $V^{\otimes r}$ is to replace the
complicated infinite dimensional group algebra $kG$ by its finite
dimensional image $\rho(kG)$. Obviously one does not lose information
about the module structure by doing this. The algebra $\rho(kG)
= \End_{\Sym_r}(V^{\otimes r})$ is isomorphic to the Schur algebra
$S_k(n,r)$ defined in \cite{Green:2007a}; its module category (when
$k$ is infinite) is equivalent to the category of homogeneous
polynomial $kG$-modules of degree $r$.

Our last remark is on the relevance of polynomial functor technology,
which in the case of infinite fields yields strong results by the
methods of Friedlander and Suslin \cite{FS} and Franjou, Friedlander,
Scorichenko and Suslin \cite{FFSS}. Referring to Section 2 of
\cite{FFSS} and Section 1 of \cite{FS}, the problem for finite fields
is that polynomial functors and strict polynomial functors are
different. Indeed, strict polynomial functors of a given degree always
have finite projective resolutions, while polynomial functors
sometimes do not. The methods of \cite{FFSS,FS} produce results for
strict polynomial functors, which gives strong information for
infinite fields. But for example if we look at Corollary 3.13 of
\cite{FS} we might be tempted to suppose that our Theorem
\ref{th:main1} can be extended to the case $n=r$, which it cannot; see
the discussion preceding the theorem.

\section{Generalization of the Carter--Lusztig result}\noindent
If $k$ is an infinite field, the map $\rho$ described in the first
paragraph of the introduction is surjective by the theorem of De
Concini and Procesi. However, over a finite field this may no longer
be true. For example, if $k=\F_2$ and $n=r=2$ then
$G=GL(2,2)\cong\Sym_3$, $V$ is a projective irreducible
$\F_2G$-module, and $V^{\otimes 2}$ decomposes as a direct sum of $V$
and a two dimensional module with a one dimensional fixed space; the
latter is isomorphic to the projective cover of the trivial module. It
follows that $\End_G(V^{\otimes 2})$ is three dimensional while
$k\Sym_2$ is two dimensional, so $\rho$ cannot be surjective, although
it is injective. The situation is worse for $n<r$: we let the reader
check that if $k=\F_2$, $n=2$ and $r=3$ then $\rho$ is neither
injective nor surjective.

In Lemma 3.1 of \cite{Carter/Lusztig:1974a} it is proved that for
any infinite field $k$ the map $\rho$ is an isomorphism, provided that
$n \ge r$.  However, a slightly weaker statement remains true over
finite fields. We prove the following, which does not depend on
whether $k$ is finite or infinite.

\begin{theorem}\label{th:main1}
Let $k$ be an arbitrary field.  If $n \ge r$ then $\rho\colon k\Sym_r
\to \End_G(V^{\otimes r})$ is injective. If $n\ge r+1$ then $\rho$ is
an isomorphism.
\end{theorem}
\begin{proof}
Let $v_1,\dots,v_n$ be a $k$-basis for $V$.  The injectivity for $n
\ge r$ follows from the fact that the images under $\Sym_r$ of
$v_1\otimes \dots \otimes v_r$ are linearly independent.

It remains to show that $\rho$ is surjective when $n\ge r+1$, so we
assume we are in this case. Let $\phi\in\End_G(V^{\otimes r})$.  There
exist uniquely determined scalars $\lambda_{i_1,\dots,i_r}$ ($1\le
i_j\le r$) such that
\[ \phi(v_1 \otimes \dots \otimes v_r) = \sum_{i_1,\dots,i_r}
\lambda_{i_1,\dots,i_r}v_{i_1} \otimes \dots \otimes v_{i_r}. \]
This information determines the value of $\phi$ on any basis
element of the form $v_{j_1} \otimes \dots\otimes v_{j_r}$ where there
are no repetitions among $j_1,\dots, j_r$.

We claim that $\lambda_{i_1,\dots,i_r}=0$ unless $i_1,\dots,i_r$
is a permutation of $1,\dots,r$. First suppose some $i_j$ is 
strictly greater than $r$.
Consider the element $g\in G$ defined by $g(v_i)=v_i$ if
$i \ne i_j$ and $g(v_{i_j})=v_{i_j} + v_1$. Then $g$ fixes
$v_1\otimes\dots \otimes v_r$.
Since $\phi$ commutes with the action of $g$, 
it follows that $g$ fixes $\phi(v_1\otimes \dots \otimes v_r)$. 
So  we have 
\[ \lambda_{i_1,\dots,i_{j-1},1,i_{j+1},\dots,i_r} = 
\lambda_{i_1,\dots,i_{j-1},1,i_{j+1},\dots,i_r} +
\lambda_{i_1,\dots,i_{j-1},i_j,i_{j+1},\dots,i_r}. \] Notice that in
the left hand side of the above equation, as well as in the first term
on the right hand side, not only is $i_j$ replaced by $1$, but also so
is every $i_k$ which happens to satisfy $i_k=i_j$.  It follows from
this equation that $\lambda_{i_1,\dots,i_{j-1},i_j,i_{j+1},\dots,i_r}
=0$.

Next, we suppose that all the indices $i_j$ satisfy $1\le i_j \le r$,
and we show that if two of the indices are equal then
$\lambda_{i_1,\dots,i_r}=0$. By assumption, we have
$\{i_1,\dots,i_r\} \subseteq \{1,\dots,r\}$. So if two of the indices
are equal then we may choose $j$ such that $1\le j\le r$ and
$j\not\in\{i_1,\dots,i_r\}$.  Consider the element $g\in G$ defined by
$g(v_i)=v_i$ if $i\ne j$, and $g(v_j)=v_j - v_{r+1}$. This makes sense
because $n\ge r+1$.  Then
\[ g(v_1\otimes\dots\otimes v_r) = v_1\otimes\dots\otimes v_r
- v_1 \otimes\dots\otimes v_{j-1}\otimes v_{r+1} \otimes
v_{j+1} \otimes\dots \otimes v_r. \]
The coefficient of $v_{i_1} \otimes \dots \otimes v_{i_r}$ in
$g(\phi(v_1\otimes\dots\otimes v_r))$ is $\lambda_{i_1,\dots,i_r}$
whereas the coefficient in
$\phi(g(v_1\otimes \dots \otimes v_r))$ is
$\lambda_{i_1,\dots,i_r} - \lambda_{i_1,\dots,i_r} = 0$. 
Since these are supposed to be equal, 
it follows that $\lambda_{i_1,\dots,i_r}=0$.

Next, we show that the value of $\phi$ on $v_1\otimes \dots\otimes v_r$
determines the value on $v_{j_1} \otimes \dots \otimes v_{j_r}$,
where we allow repetitions among $j_1,\dots,j_r$. We do this
by induction on $m=r-|\{j_1,\dots,j_r\}|$, the case $m=0$ being clear.
Permuting the indices if necessary, 
without loss of generality we may suppose that $j_{r-1}=j_r$.
Choose $\ell$ such that $1\le \ell \le n$ and 
$\ell\not\in\{j_1,\dots,j_r\}$. 
Consider the element $g\in G$ defined by $g(v_i)=v_i$ if $i\ne \ell$
and $g(v_\ell)=v_\ell + v_{j_r}$. Then
\[ g(v_1\otimes \dots \otimes v_{j_{r-1}} \otimes v_\ell)
=  v_1\otimes \dots \otimes v_{j_{r-1}} \otimes v_\ell +
v_1\otimes \dots \otimes v_{j_{r-1}} \otimes v_{j_r}. \]
Hence 
\[ \phi(v_1\otimes \dots \otimes v_{j_r}) =
g(\phi(v_1 \otimes \dots \otimes v_{j_{r-1}} \otimes v_\ell))
- \phi(v_1 \otimes \dots \otimes v_{j_{r-1}} \otimes v_\ell) \]
which is determined by the inductive hypothesis.

The conclusion of this argument is that the map $\phi$ is determined
by the constants $\lambda_{i_1,\dots,i_r}$ with $i_1,\dots,i_r$
a permutation of $1,\dots,r$. Since every such set of constants is
the effect of some element of $k\Sym_r$, $\rho$
is surjective.
\end{proof}

\section{Schur--Weyl duality for Schur algebras}\label{sec:2}\noindent
In this expository section we summarize the known results concerning
Schur--Weyl duality between the Schur algebra and symmetric group.
Let us briefly recall the definition of the Schur algebra $S_k(n,r)$,
which makes sense for any commutative ring $k$.  First, consider the
polynomial algebra $k[x_{ij}]$ in $n^2$ commuting variables $x_{ij}$,
$1 \le i,j \le n$. This is naturally graded by regarding each
generator as having degree 1; denote by $A_k(n,r)$ the $r$th graded
component. Thus $k[x_{ij}] = \bigoplus_{r \ge 0}\, A_k(n,r)$. Now,
$k[x_{ij}]$ has a natural coalgebra structure with comultiplication
$\Delta$ the algebra homomorphism given on generators by $x_{ij} \to
\sum_{l=1}^n x_{il} \otimes x_{lj}$, and counit given by $x_{ij} \to
\delta_{ij}$ (Kronecker's delta). Since comultiplication is an algebra
homomorphism, $k[x_{ij}]$ is a bialgebra. One easily checks that
$A_k(n,r)$ is a subcoalgebra of $k[x_{ij}]$, for every $r \ge 0$.
\begin{defn}
  The Schur algebra is the algebra $S_k(n,r) := A_k(n,r)^*$. 
\end{defn}
\noindent
Note that it is a well known general fact that the linear dual $C^* =
\Hom_k(C,k)$ of any $k$-coalgebra $C$ is an algebra in a natural
way. The multiplication on $C^*$ comes from the comultiplication on
$C$, by the rule $ff' = (f \overline{\otimes} f')\Delta$, for all
$f,f' \in C^*$. (As usual, $f \overline{\otimes} f'$ stands for the map
$a \otimes a' \mapsto f(a) f'(a')$.) 

Now let $I(n,r)$ denote the set $\{1,\dots,n\}^r$ of $r$-tuples of
elements of $\{1,\dots,n\}$. For $\ii = (i_1, \dots, i_r)$, $\jj = (j_1,
\dots, j_r)$ in $I(n,r)$ define 
\[ x_{\ii,\,\jj} :=  x_{i_1,\,j_1} \cdots x_{i_{r-1},\,j_{r-1}}\, x_{i_r,\,j_r}. \]
The symmetric group $\Sym_r$ acts naturally by place-permutation on
$I(n,r)$, and thus acts on $I(n,r) \times I(n,r)$, and $x_{\ii,\,\jj}
= x_{\ii',\,\jj'}$ if and only if $(\ii,\jj)$ and $(\ii',\jj')$ lie in
the same orbit. Thus if $\Omega$ is a set of representatives for the
orbits then $\{ x_{\ii,\,\jj} \colon (\ii,\jj) \in \Omega \}$ is a basis of
$A_k(n,r)$.  Write $(\ii,\jj) \sim (\ii',\jj')$ if $(\ii,\jj)$ and
$(\ii',\jj')$ lie in the same orbit.  For $\ii, \jj \in I(n,r)$ define
$\xi_{\ii,\,\jj} \in S_k(n,r)$ by
\[ \xi_{\ii,\,\jj}(x_{\ii',\,\jj'}) = 
\begin{cases}
  1 & \text{ if } (\ii,\jj) \sim (\ii',\jj') \\
  0 & \text{ otherwise.}
\end{cases}
\]
Then of course $\xi_{\ii,\,\jj} = \xi_{\ii',\,\jj'}$ if and only if
$(\ii,\jj) \sim (\ii',\jj')$, and $\{ \xi_{\ii,\,\jj} \colon (\ii,\jj) \in
\Omega \}$ is a basis of $S_k(n,r)$. Now we define an action of
$S_k(n,r)$ on $V^{\otimes r}$ by the rule
\[ \xi \cdot v_\jj = \sum_{\ii \in I(n,r)} \xi(x_{\ii,\,\jj})\, v_\ii \]
for any $\xi \in S_k(n,r)$, where $v_\ii := v_{i_1}\otimes \cdots
\otimes v_{i_r}$ for any $\ii \in I(n,r)$.  It is 
straightforward to check that this action commutes with the action of
$\Sym_r$. Then an elementary direct computation \cite[Theorem
  (2.6c)]{Green:2007a} shows that
\begin{equation}\label{Schuriso}
  S_k(n,r) \cong \End_{\Sym_r}(V^{\otimes r}).
\end{equation}
The argument given there is valid for any commutative ring $k$.

Since the actions of $S_k(n,r)$ and $\Sym_r$ on $V^{\otimes r}$
commute, one has also a natural algebra homomorphism
\begin{equation}\label{Bryantrk}
  k\Sym_r \to \End_{S_k(n,r)}(V^{\otimes r}).
\end{equation}
Assume now that $k$ is a field. If $k$ is infinite then of course the
map \eqref{Bryantrk} is surjective, because $\varphi(kG) \cong
S_k(n,r)$ and the action of $kG$ induces the action of $S_k(n,r)$, so
$\End_{S_k(n,r)}(V^{\otimes r}) = \End_{G}(V^{\otimes r})$ and thus the map
in \eqref{Bryantrk} may be identified with the map $\rho$.  But the
surjectivity statement also holds in case $k$ is finite, even though
the isomorphism $\varphi(kG) \cong S_k(n,r)$ may fail. This has been
shown by an elementary argument of Bryant \cite[Lemma
  2.4]{Bryant:inf}. So the map \eqref{Bryantrk} is surjective for any
field $k$, finite or infinite. This is needed in the proof of Theorem
\ref{Thrall}. We note that in Bryant's argument it is enough to
assume the surjectivity of $\rho$ in case $k$ is algebraically closed.

\section{Large enough finite fields}\noindent
In this section we prove Schur--Weyl duality over finite fields which
are sufficiently large, assuming the result is known in case the field
is algebraically closed. The central issue is whether or not the
Schur algebra $S_k(n,r)$ is isomorphic to the subalgebra of
$\End_k(V^{\otimes r})$ generated by the image of the
$G=GL(V)$-action. This holds when $k$ is infinite, but in fact we will
show it holds whenever $k$ has strictly more than $r$ elements, as an
application of the Chevalley group construction. Our argument is a
refinement of an argument given in Section 3.1 of
\cite{Carter/Lusztig:1974a}.

We will need to pay close attention to change of scalars in this
section; in particular, we are going to write $V_k$ instead of $V$ in
order to emphasize the field of scalars.

Consider the complex Lie group $GL(n,\C)$ and its Lie algebra
$\mathfrak{g}:= \mathfrak{gl}_n(\C)$.  Let $U = U_\C(\mathfrak{g})$ be
the universal enveloping algebra (over $\C$) of $\mathfrak{g}$. One
may regard $U$ as the associative $\C$-algebra given by generators
$e_{ij}$, $1 \le i,j \le n$ with the relations
\[
  e_{ij} e_{ab} - e_{ab} e_{ij} = \delta_{aj}e_{ib} - \delta_{ib} e_{aj}
\]
for all $1 \le i,j,a,b \le n$. (The $e_{ij}$ correspond to the matrix
units in the Lie algebra.) Let $U'$ be the subalgebra of $U$ generated
by the $e_{ij}$ for $i \ne j$. Then $U'$ may be identified with the
universal enveloping algebra of $\mathfrak{sl}_n(\C)$.

Let $V$ be an $n$-dimensional complex vector space, with basis $\{v_1,
\dots, v_n\}$. Define an action of $\mathfrak{g}$ on $V$ by $e_{ij}
v_a = \delta_{aj} v_i$. Identifying $V$ with $\C^n$ by means of the
basis, one sees that the action is just by matrix multiplication. Thus
$V^{\otimes r}$ is for any $r \ge 0$ a $\mathfrak{g}$-module and hence
a $U$-module.  We regard $V^{\otimes r}$ as a $U'$-module by
restriction.  Let $U'_\Z$ be the subring of $U$ generated by all
divided powers
\[
\frac{e_{ij}^m}{m!}\quad (1\le i \ne j \le n,\ m\ge 0) 
\]
and let $U_\Z$ be the subring generated by $U'_\Z$ and all elements of
the form
\[
\binom{e_{ii}}{m} := \frac{e_{ii}(e_{ii}-1) \cdots (e_{ii}-m+1)}{m!}
\quad (1\le i \le n,\ m \ge 0).
\]
Then $U_\Z$ (resp., $U'_\Z$) is the Kostant $\Z$-form of $U$ (resp.,
$U'$). Let $V_\Z = \sum \Z v_i$ and note that $V_\Z$ is stable under
the action of $U_\Z$. It follows that we have for any $r\ge 0$ a
$U_\Z$-module structure on $V_\Z^{\otimes r}$, and hence also a
$U'_\Z$-module structure.  Thus for any commutative ring $k$ we have a
$k\otimes_\Z U_\Z$-module structure on $k\otimes_\Z V_\Z^{\otimes
  r}$. Set $U_k = k \otimes_\Z U_\Z$, $U'_k = k \otimes_\Z U'_\Z$.
Then $U_k$ (resp., $U'_k$) is the algebra of distributions on the
algebraic group scheme $\mathbf{GL}_{n,k}$ (resp.,
$\mathbf{SL}_{n,k}$). We identify $k\otimes_\Z V_\Z^{\otimes r}$ with
$V_k^{\otimes r}$ as usual, where $V_k := k \otimes_\Z V_\Z$. Note
that $U_\C \cong U$, $U'_\C \cong U'$ as algebras, and $V_\C \cong V$
as vector spaces. Moreover, the $U$-module structure on $V$ is
isomorphic with the $U_\C$-module structure on $V_\C$, so that when
$k=\C$ these constructions recover the original objects.

So we have a $U_k$-module structure on $V_k^{\otimes r}$. It is easy
to check that on $V_k^{\otimes r}$ the action of $U_k$ commutes with
the place-permutation action of the symmetric group $\Sym_r$. Thus we
have algebra homomorphisms
\begin{equation}
\psi\colon U_k \to \End_{\Sym_r}(V_k^{\otimes r});\qquad 
\psi'\colon U_k \to \End_{\Sym_r}(V_k^{\otimes r})
\end{equation}
where the second map is just the restriction of the first.  For our
purposes, the following simple observation is key.

\begin{lemma}
For any $m > r$ the elements 
\[
\frac{e_{ij}^m}{m!},\ \binom{e_{ii}}{m}\ \in U \qquad (1 \le i\ne
j \le n,\ m \ge 0)
\] 
act as zero on $V^{\otimes r}$; hence for any $m > r$ the elements
\[
1 \otimes \frac{e_{ij}^m}{m!},\ 1 \otimes \binom{e_{ii}}{m}\ \in
U_k \qquad (1 \le i\ne j \le n,\ m \ge 0)
\]
act as zero on $V_k^{\otimes r}$.
\end{lemma}

\begin{proof}
The weight of a simple tensor $v_\jj$ (for $\jj = (j_1, \dots, j_r)
\in I(n,r)$) is the composition $\lambda = (\lambda_1, \dots,
\lambda_n)$ where $\lambda_i$ counts the number of $a \in [1,r]$ such
that $j_a = i$, for each $i$.  The weight $\lambda$ is an $n$-part
composition of $r$ (i.e., $\lambda_1 + \cdots + \lambda_n = r$).  One
checks that $e_{ij}$ for $i \ne j$ sends a simple tensor of weight
$\lambda$ onto
\[
\begin{cases}
\text{a simple tensor of weight } \lambda+\varepsilon_i-\varepsilon_j &
\text{if } \lambda_j \ge 1 \\ 0 & \text{otherwise}.
\end{cases}
\]
Here $\{ \varepsilon_1, \dots, \varepsilon_n \}$ is the standard basis
of $\Z^n$.  Now by induction on $m$ it follows that
$\frac{e_{ij}^m}{m!}$ takes a simple tensor of weight $\lambda$ onto
\[
\begin{cases}
\text{a simple tensor of weight } \lambda+m\varepsilon_i-m\varepsilon_j &
\text{if } \lambda_j \ge m \\ 0 & \text{otherwise}.
\end{cases}
\]
When $m > r$, no composition $\lambda$ of $r$ satisfies $\lambda_j \ge
m$, no matter what $j$ we look at. Hence $\frac{e_{ij}^m}{m!}$ acts as
zero on all simple tensors in $V^{\otimes r}$. This proves the claim
for $i \ne j$. 

It remains to show the $\binom{e_{ii}}{m}$ also act as zero. For this
note that $e_{ii}$ acts as the scalar $\delta_{i,\,j_1} + \cdots +
\delta_{i,\,j_r}$ on the simple tensor $v_\jj = v_{j_1} \otimes \cdots
\otimes v_{j_r}$. This scalar is an integer in the interval $[0,r]$,
so for any $m>r$ the element $\binom{e_{ii}}{m}$ acts as zero.
\end{proof}

Assume from now on that $k$ is a field. For any $t \in k$, $1 \le i,j
\le n$ we define elements $E_{ij}(t) \in \End(V_k)$ by the rule
\[
  E_{ij}(t) = 1 + t \psi(1\otimes e_{ij}).
\]
Then from the lemma it follows that for $i \ne j$, $E_{ij}(s)
E_{ij}(t) = E_{ij}(s+t)$ for all $s,t \in k$. Thus $E_{ij}(t)$ is for
$i \ne j$ invertible, i.e., $E_{ij}(t) \in GL(V_k)$. One easily checks
that $E_{ii}(t) \in GL(V_k)$ provided that $t \ne -1$. Now it is clear
that $GL(V_k)$ is generated by the elements $E_{ij}(t)$ ($1 \le i \ne
j \le n$, $t \in k$) along with the elements $E_{ii}(t)$ ($1 \le i \le
n$, $-1 \ne t \in k$). Moreover, $SL(V_k)$ is generated by the
$E_{ij}(t)$ ($1 \le i \ne j \le n$, $t \in k$). Of course, $GL(V_k)$
acts naturally on $V_k$, so $GL(V_k)$ acts on $V_k^{\otimes r}$; this
is the action already considered in the introduction. By restriction,
we also have an action of $SL(V_k)$ on $V_k^{\otimes r}$. As already
noted, these actions commute with the place-permutation action of
$\Sym_r$, so we have algebra homomorphisms
\begin{equation}
\varphi\colon kGL(V_k) \to \End_{\Sym_r}(V_k^{\otimes r}); \qquad
\varphi'\colon kSL(V_k) \to \End_{\Sym_r}(V_k^{\otimes r}).
\end{equation}

\begin{lemma}
Let $k$ be a field. 
\begin{enumerate}
\item[(i)] If the order of $k$ is strictly larger
than $r+1$, then $\varphi$ is surjective if and only if $\psi$ is
surjective. 

\item[(ii)] If the order of $k$ is strictly larger
than $r$, then $\varphi'$ is surjective if and only if $\psi'$ is
surjective.
\end{enumerate}
\end{lemma}

\begin{proof}
First, we claim that as operators on $V_k^{\otimes r}$ the maps $\psi$
and $\varphi$ are related by the following, for $t \in k$ and $i \ne
j$:
\begin{equation}\label{a}
  \varphi(E_{ij}(t)) = \sum_{m=0}^r t^m \psi(1\otimes
  \frac{e_{ij}^m}{m!}).
\end{equation}
This may be checked directly; see formula (25) of
\cite{Carter/Lusztig:1974a}. 

Choose $r+1$ distinct values $t_0, t_1, \dots, t_r$ for $t$ in
$k$. (This is possible because $k$ has more than $r$ elements, by
hypothesis.) This results in a system of equations
\begin{align*}
\varphi(E_{ij}(t_0)) &= 1 + t_0\psi(1\otimes e_{ij}) + t_0^2
\psi(1\otimes \frac{e_{ij}^2}{2!}) + \cdots + t_0^r
\psi(1\otimes \frac{e_{ij}^r}{r!})\\
\varphi(E_{ij}(t_1)) &= 1 + t_1\psi(1\otimes e_{ij}) + t_1^2
\psi(1\otimes \frac{e_{ij}^2}{2!}) + \cdots + t_1^r
\psi(1\otimes \frac{e_{ij}^r}{r!})\\
\vdots \hspace{0.3in} & \hspace{1.2in} \vdots  \hspace{1.6in} \vdots\\
\varphi(E_{ij}(t_r)) &= 1 + t_r\psi(1\otimes e_{ij}) + t_r^2
\psi(1\otimes \frac{e_{ij}^2}{2!}) + \cdots + t_r^r
\psi(1\otimes \frac{e_{ij}^r}{r!})
\end{align*}
whose coefficient matrix is a Vandermonde matrix with nonzero
determinant, hence invertible. Thus there exist scalars $a_{ml}$
($0 \le m,l \le r$) in $k$ such that 
\begin{equation}\label{psimap}
  \psi(1\otimes \frac{e_{ij}^m}{m!}) = \sum_{l=0}^r a_{ml}
  \varphi(E_{ij}(t_l))
\end{equation}
for each $m$.  Thus the generators $\psi(1\otimes
\frac{e_{ij}^m}{m!})$ of the image of $\psi'$ all lie in the image of
$\varphi$ on $SL(V_k)$, proving that $\im \varphi' \subseteq \im
\psi'$. The reverse containment is clear from \eqref{a}, so we have
proved part (ii).

It remains to consider part (i). For this we must consider in addition
the operators $E_{ii}(t)$. We claim that
\begin{equation}\label{b}
\varphi( E_{ii}(t) ) = \sum_{m=0}^r t^m \psi(1 \otimes \binom{e_{ii}}{m}) 
\end{equation}
for any $-1 \ne t \in k$, and any $i$. Again this may be checked
directly, as in (25) of \cite{Carter/Lusztig:1974a}. We can now repeat
the above argument, and conclude that so long as $k$ contains at least
$r+1$ elements different from $-1$, the generators $\psi(1\otimes
\binom{e_{ii}}{m})$ of the image of $\psi$ all lie in the image of
$\varphi$ on $GL(V_k)$, so we have $\im \psi \subseteq \im \varphi$.
The opposite inclusion is clear from \eqref{a} and \eqref{b}, so part
(i) is proved.
\end{proof}

\begin{theorem}\label{Thrall}
If the order of the field $k$ is strictly larger than $r$, then both
$\varphi$ and $\varphi'$ are surjective.
\end{theorem}

\begin{proof}
For $k = \overline{k}$ algebraically closed the statement is known. In
that case, the surjectivity of $\varphi'$ follows from the
surjectivity of $\varphi$, since (for $k$ algebraically closed)
$GL(V_k)$ is generated by $SL(V_k)$ and the scalar operators $c\,
\mathrm{id}_{V_k}$ ($0 \ne c \in k$).  As pointed out by Carter and
Lusztig \cite[p.~209]{Carter/Lusztig:1974a} this implies that the
natural maps
\begin{equation} \label{psiZ}
  \psi_\Z\colon U_\Z \to \End_{\Sym_r}(V_\Z^{\otimes r}); \qquad 
  \psi'_\Z\colon U'_\Z \to \End_{\Sym_r}(V_\Z^{\otimes r})
\end{equation}
are both surjective. The argument is as follows. First, apply the
preceding lemma to conclude that $\psi$, $\psi'$ are surjective. Then
apply the fact that if $f\colon A \to A'$ is a homomorphism between two
free $\Z$-modules of finite rank, then $f$ is surjective if and only if
$1 \otimes f\colon k \otimes_\Z A \to k \otimes_\Z A'$ is surjective for
every algebraically closed field $k$.

Then we obtain the surjectivity of $\psi = 1 \otimes \psi_\Z$, $\psi'
= 1 \otimes \psi'_\Z$ by right exactness of tensor product. Thus it
follows that $\psi, \psi'$ are surjective for any field $k$. (In fact,
this holds for any commutative ring.)  Now we again apply the
preceding lemma to conclude that $\varphi'$ is surjective provided
$|k| > r$. Since $\varphi'$ is the restriction of $\varphi$ this
implies the surjectivity of $\varphi$ as well.
\end{proof}

\begin{cor}\label{cor:main2}
If the field $k$ has strictly more than $r$ elements, then the natural
maps $\varphi\colon kGL(V_k) \to \End_{\Sym_r}(V_k^{\otimes r})$,
$\rho\colon k\Sym_r \to \End_{GL(V_k)}(V_k^{\otimes r})$ are
surjective.  The same statement holds if $GL(V_k)$ is replaced by
$SL(V_k)$.
\end{cor}

\begin{proof}
The surjectivity of $\varphi$ is the statement of the preceding
theorem.  It follows (setting $G = GL(V_k)$) that
$\End_{G}(V_k^{\otimes r}) = \End_{\varphi(kG)}(V_k^{\otimes
  r})= \End_{S_k(n,r)}(V_k^{\otimes r}) = \rho(k\Sym_r)$, where the
last equality is by the surjectivity statement of
\eqref{Bryantrk}. This proves the surjectivity of $\rho$, so the first
claim is proved. The proof of the second claim is similar.
\end{proof}

We note the following improvement of \cite[Theorem
  3.1(i)]{Carter/Lusztig:1974a}.

\begin{cor}
  The natural map $\Z \Sym_r \to \End_{U_\Z}(V_\Z^{\otimes r})$ is
surjective (for any $n,r$). The same statement holds with $U'_\Z$ in
place of $U_\Z$.
\end{cor}

\begin{proof}
Let $k = \overline{k}$ be algebraically closed. Then the corresponding
map $k \Sym_r \to \End_{U_k}(V_k^{\otimes r})$ is surjective, by
\eqref{Bryantrk}, since $$\End_{U_k}(V_k^{\otimes r})
= \End_{\psi(U_k)}(V_k^{\otimes r}) \cong
\End_{S_k(n,r)}(V_k^{\otimes r}).$$  This implies the claim about $U_\Z$. 
The argument with $U'_\Z$ in place of $U_\Z$ is similar, since
we now know that $\psi(U_k) = \psi(U'_k)$.
\end{proof}

\begin{rks}
(a) From the preceding result and the first surjectivity statement in
  \eqref{psiZ} it follows by the right exactness of tensor product
  that the natural maps
\[
  U_k \to \End_{\Sym_r}(V_k^{\otimes r}); \quad 
  k\Sym_r \to \End_{U_k}(V_k^{\otimes r})
\]
are surjective, for any commutative ring $k$. Similarly one sees
immediately that this holds with $U'_k$ in place of $U_k$.

(b) Our proof of Theorem \ref{Thrall} is valid assuming only the
surjectivity of the maps $\rho$, $\varphi$ is known in case $k$ is
algebraically closed. That surjectivity is known by
\cite{deConcini/Procesi:1976a} and \cite{Green:2007a}.

(c) If $k=\F_2$ and $V_k$ has dimension 2, the dimension of $kGL(V_k)$
is 6 while the dimension of $\End_{\Sym_2}(V_k \otimes V_k)$ is
10. This shows that the bound $|k| > r$ in Theorem \ref{Thrall} cannot
be improved. It also shows that the bound in Theorem III of
\cite{Thrall} is incorrect. Note that, apart from the bound, the first
assertion in Theorem \ref{Thrall} is the same as Theorem III of
\cite{Thrall}.
\end{rks}

\section{An example}\noindent
In this section, we take $k = \F_q$ with $q = p^e$, and once again
write $V$ for the $n$-dimensional space $V_k$.  We will show that some
restriction on $q$ with respect to $r$ is necessary in order for
Schur--Weyl duality to hold.

\begin{theorem}
In case $n= \dim_k V = 2$ and $q=p$ is a prime, for $r$ sufficiently
large the map $\rho\colon k\Sym_r\to\End_{GL(V)}(V^{\otimes r})$ is not
surjective.
\end{theorem}
\begin{proof}
First we note that $GL(V)=GL(2,p)$ is a finite group
whose Sylow $p$-subgroups are cyclic. It follows that there are a finite
number of isomorphism classes of indecomposable $\F_pGL(V)$-modules.
Let $d$ be the sum of the dimensions of the indecomposable
$\F_pGL(V)$-modules, one from each isomorphism class. Then by a
version of the pigeon hole principle, for all finite dimensional
$\F_pGL(V)$-modules $M$ there is a summand consisting of at least
$(\dim_{\F_p}M/d) -1$ copies of a single indecomposable. Thus
$\dim_{\F_p}\End_{GL(V)}(M) \ge ((\dim_{\F_p}M/d)-1)^2$. Applying
this with $M=V^{\otimes r}$ we have
\begin{equation}\label{eq:dimEnd} 
\dim_{\F_p}\End_{GL(V)}(V^{\otimes r}) \ge \left(\frac{2^r}{d} - 1\right)^2.
\end{equation}
On the other hand, the image of $\F_p\Sym_r$ in
$\End_{GL(V)}(V^{\otimes r})$ is the Temper\-ley--Lieb algebra $T_r$, whose
dimension is the Catalan number $\binom{2r}{r}/(r+1)$. Using Stirling's
formula, asymptotically  we get
\begin{equation}\label{eq:dimTr} 
\dim_{\F_p} T_r \sim \frac{2^{2r}}{\sqrt{\pi}\,r^{3/2}}. 
\end{equation}
Comparing \eqref{eq:dimEnd} with \eqref{eq:dimTr}, we see that for 
$r$ large we have 
\begin{equation*} 
\dim_{\F_p}\End_{GL(V)}(V^{\otimes r}) > \dim_{\F_p}T_r. 
\qedhere
\end{equation*}
\end{proof}

\bibliographystyle{amsplain}
\bibliography{swdff.bbl}

\end{document}